\documentclass[12pt,leqno]{article}
\usepackage{amsfonts,amsmath,latexsym, amssymb}
\usepackage[dvips]{graphics}
\usepackage{graphics}
\usepackage{graphicx}

\newcounter{theorem}

\numberwithin{equation}{section}

\newtheorem{theorem}{Theorem}[section]
\newtheorem{corollary}[theorem]{Corollary}
\newtheorem{definition}[theorem]{Definition}

\newtheorem{lemma}[theorem]{Lemma}

\newtheorem{proposition}[theorem]{Proposition}

\newenvironment{proof}[1][Proof]{\textbf{#1.} }
{\ \Box\smallskip}

\newcommand{\RR}{\mathbb{R}}

\newcommand{\half}{\frac{1}{2}}

\newcommand{\real}{\mathbb{R}}

\newcommand{\nat}{\mathbb{N}}

\newcommand{\A}{\mathcal{A}}

\newcommand{\cP}{{\cal P}}
\newcommand{\cV}{{\cal V}}
\newcommand{\cW}{{\cal W}}
\newcommand{\cU}{{\cal U}}

\newcommand{\uone}{u_{\tau}^1}
\newcommand{\uzero}{u_{\tau}^0}

\newcommand{\uktau}{u_{\tau}^k}

\newcommand{\xiktau}{\xi_{\tau}^k}
\newcommand{\fktau}{f_{\tau}^k}

\newcommand{\uktauminus}{u_{\tau}^{k-1}}
\newcommand{\uktauminuss}{u_{\tau}^{k-2}}

\newcommand{\uitau}{u_{\tau}^i}

\newcommand{\uitauminus}{u_{\tau}^{i-1}}

\newcommand{\untau}{u_{\tau}^n}

\newcommand{\xintau}{\xi_{\tau}^n}
\newcommand{\fntau}{f_{\tau}^n}

\newcommand{\untauminus}{u_{\tau}^{n-1}}
\newcommand{\untauminuss}{u_{\tau}^{n-2}}

\newcommand{\skalar}[1]{\left\langle #1 \right\rangle}

\newcommand{\dual}[2]{\left\langle #1 \right\rangle_{{#2}^*\times #2}}

\begin{document}

\title{\bf Convergence of double step scheme for a class of parabolic Clarke subdifferential inclusions  \thanks{\ Research supported by European Union's Horizon 2020 Research and Innovation Programme
		under the Marie Sklodowska-Curie Grant Agreement No 823731 CONMECH. }}

\author{
	Krzysztof Bartosz$^{\,1}$,  \
	Pawe{\l} Szafraniec$^{\,1}$, \ Jing Zhao$^{\,2}$
	\\ ~ \\
	{\small $^1$ Jagiellonian University, Faculty of Mathematics and Computer Science} \\
	{\small ul. \L{}ojasiewicza 6, 30348 Krakow, Poland} \\
{\small$^2$ College of Sciences, Beibu Gulf University,}\\
	 {\small Qinzhou, Guangxi 535000, China}	
}

\date{}
\maketitle \thispagestyle{empty}

\begin{center}
{\it The paper is dedicated to Professor Stanis\l{}aw Mig\'orski} \\
{\it on the occasion of his 60th birthday.}
\end{center}
\vskip 4mm

\noindent {\footnotesize{\bf Abstract.} 
In this paper we deal with a first order evolution inclusion involving a multivalued term generated by a Clarke subdifferential of a locally Lipschitz potential.  For this problem we construct a double step time-semidiscrete approximation, known as the Rothe scheme. We study a sequence of solutions of the semidiscrete approximate problems and provide its weak convergence to a limit element that is a solution of the original problem. 

\vskip 2mm










\vskip 12mm
\section{Introduction}\label{Sec_1}
We consider a class of evolutionary inclusions of parabolic type. A multivalued term, which appears in our problem, has a form of Clarke subdifferential of a locally Lipschitz function. Such problems are closely related to so called hemivariational inequalities (HVIs) in the sense that the solution of HVI can be found by solving a corresponding inclusion, as is presented in Section \ref{sec_example}. HVIs are generalization of variational inequalities and they play very important role in modelling of various problems arising in mechanics, physics, and the engineering sciences.   

The theory of HVIs has been introduced by Pangiotopulos in 1980s \cite{P} in order to describe several important mechanical and engineering problems with nonmonotone phenomena in solid mechanics. The concept of HVIs is based on the notion of Clarke subdifferential of a locally Lipschitz functional that may be nonconvex - see \cite{Clarke}. The theory of HVIs has been rapidly developed in the last decades as is illustrated in \cite{DMP1, HMSBOOK, MOSBOOK} for example. For more recent results on theory and applications of HVIs see \cite{Rev2, Rev1}. In addition to the theory, numerical aspects of discussed problems began to be developed especially in the last few years. They involve various kinds of discrete schemes based on temporal and spatial discretization.
In our paper we deal with the first kind of discretization, known as Rothe method. We are motivated by \cite{Kalita2013}, in which Rothe method was used for the first time to solve an evolutionary inclusion involving Clarke subdifferential. In the quoted paper the time discretization technique based on the implicit (backward) Euler scheme was used. In this approach the time derivative of unknown function $u$ is approximated by a finite difference according to the role $u'(t_n)\simeq\frac{1}{\tau}(u^{n}-u^{n-1})$, where $\tau$ represents the time step. The result presented there has been developed in \cite{Kalita} by applying more general $\theta$-scheme. Both results deal with evolutionary inclusions of parabolic type. In \cite{BCKYZ, Bartosz_Sofonea_1}, Rothe method has been applied for more general parabolic problem, namely variational-hemivariational inequalities.  The techniques introduced  in \cite{Kalita2013, Kalita} has been successfully adapted to a second order evolutionary inclusion in Chapter 5 of \cite{HMSBOOK} and in \cite{Bartosz_1, Bartosz_2}.  We also refer to \cite{Han_J_Migorski_Zeng, Liu_Zeng_Migorski, Migorski_Zeng_1, Migorski_Zeng_2,  Peng_Xiao} for more result concerning Rothe method in analysis of various kinds of evolutionary hemivariational inequalities.\\
In this paper we deal with the parabolic Clarke inclusion which has been already studied in \cite{Kalita2013}. But this time we use two steps backward differential formula (BDF) based on the idea that the time derivative $u'(t_n)$ is approximated by the derivative of the unique second order polynomial that interpolates the approximate solution in three points, namely
\begin{align*}
u'(t_n)\simeq\frac{1}{\tau}\left(\frac{3}{2}\untau-2\untauminus+\frac{1}{2}\untauminuss\right).
\end{align*}
Hence, it seems to be more accurate then the standard implicit Euler scheme used in \cite{Kalita2013}, which involves only two points to approximate the time derivative.    
On the other hand, two steps BDF scheme is more complicated and its analysis requires more sophisticated technique than the one used in previous related papers. Our approach relies on techniques from \cite{Emmrich}, where BDF for non-Newtionian fluids was considered. Nevertheless we think that our efforts are fruitful and the presented result opens new perspectives in numerical analysis of nonlinear evolutionary inclusions.
 
The rest of the paper is structured as follows. In Section \ref{Sec_2} we introduce preliminary materials and recall basic results to be used letter.  In Section \ref{Sec_3} we formulate the problem that is an abstract evolution inclusion of parabolic type involving Clarke subdifferential. We also list all assumptions on the data of the problem. In Section \ref{Sec_4} we introduce a time semidiscrete Rothe scheme corresponding to the original problem and provide the existence result for the approximate one. Moreover we derive a-priori estimates for the approximate solution. In Section \ref{Sec_5} we provide the convergence result, which shows that the sequence of solutions of semidiscrete problem converges weakly to the solution of the original problem as the discretization parameter converges to zero. Finally, in Section \ref{sec_example}, we provide a simple example of parabolic boundary problem for which our theoretical result can be applied.

\section{Notation and preliminaries}\label{Sec_2}
In this section we introduce notation and recall several known results that will be used in the rest of the paper.

Let $X$ be a real normed space. Everywhere in the paper we will use the symbols $\|\cdot\|_X$, $X^*$ and $\dual{\cdot,\cdot}{X}$ to denote the norm in $X$, its dual space and the duality pairing of $X$ and $X^*$, respectively. Moreover, if $Y$ is a normed space and $f\in {\cal L}(X,Y)$, we will briefly write $\|f\|$ instead of $\|f\|_{{\cal L}(X,Y)}$ and we will use notation $f^*\colon Y^*\to X^*$ for the adjoint operator to $f$. 
We start with the definition of the Clarke generalized directional derivative and the Clarke subdifferential.

\begin{definition}\label{Def_Clarke_subdifferential}
Let $\varphi\colon X\to \RR$ be a locally Lipschitz function. The
Clarke generalized directional derivative of $\varphi$ at the
point $x\in X$ in the direction $v\in X$,  is defined by
\[ \varphi^0(x;v)=\limsup_{y\to x, \lambda \downarrow 0}
\frac{\varphi(y+\lambda v)-\varphi(y)}{\lambda}.\]
The Clarke subdifferential of $\varphi$ at
$x$ is a subset of $X^*$ given by
\[ \partial \varphi(x)=\{\,\xi\in X^*\,\,|\,\, \varphi^0(x;v)\geq
\dual{\xi,v}{X}\,\,\,\,\, \text{for all }\,\,  v\in X\,\}. \]
\end{definition}

In what follows, we recall the definition of pseudomonotone operator in both single-valued and multivalued cases.

\begin{definition}\label{Def_pseudomonotone_single}
 A single-valued operator $A\colon X\to X^*$ is called pseudomonotone
 if for any sequence $\{v_n\}_{n=1}^\infty\subset X$, $v_n\to v$
 weakly in $X$ and $\limsup_{n\to\infty}\dual{Av_n, v_n-v}{X}\leq
 0$
 imply that
 $\dual{Av, v-y}{X} \leq \liminf_{n\to\infty}\dual{Av_n, v_n-y}{X}$
 for every $y\in X$.
\end{definition}

\begin{definition}\label{Def_pseudomonotone_multi}
	A multivalued operator $A\colon X\to 2^{X^*}$ is called pseudomonotone if the following conditions hold:
	\begin{itemize}
		\item[ 1)] $A$ has values which are nonempty, bounded, closed and convex.
		\item[ 2)] $A$ is upper semicontinuous (usc, in short) from every finite dimensional subspace of $X$
		into $X^{\ast}$ endowed with the weak to\-po\-lo\-gy.
		\item[ 3)] For any sequence $\{v_n\}_{n=1}^\infty\subset X$ and
		any $v_n^{\ast} \in A(v_n)$, $v_n \rightarrow v$ weakly in $X$ and\\ $ \limsup_{n\rightarrow
			\infty} \dual{v_n^\ast, v_n - v}{X}\leq 0$ imply that for
		every $y \in X$ there exists $u(y) \in A(v)$ such that $\dual{
			u(y), v - y}{X}
		\leq \liminf_{n \rightarrow \infty} \dual{ v_n^{\ast}, v_n - y}{X}$.
	\end{itemize}
\end{definition}

The following two propositions provide an important class of pseudomonotone operators that will appear in the next section.
They correspond to Proposition 5.6 in \cite{HMSBOOK} and Proposition
1.3.68 in \cite{DMP2}, respectively.

\begin{proposition}\label{Prop_2.1}
	Let $X$ and $U$ be two reflexive Banach spaces and $\iota\colon  X\to
	U$ a linear, continuous and compact operator. Let $J\colon
	U\to\real$ be a locally Lipschitz functional and  assume that its Clarke
	subdifferential satisfies
	\begin{equation}\nonumber
	\|\xi\|_{U^*}\leq c(1+\|u\|_U)\quad\forall u\in U,\,\,\xi\in\partial J(u)
	\end{equation}
	with $c>0$. Then the multivalued operator $M\colon X\to 2^{X^*}$ defined by
	\begin{equation}\nonumber
	M(v)=\iota^*\partial J(\iota v) \quad \text{for all}\,\,\,  v\in X
	\end{equation}
	is pseudomonotone.
\end{proposition}

\begin{proposition}\label{Prop_sum_pseudo}
	Assume that  $X$ is a reflexive Banach space and $A_1, A_2\colon X\to
	2^{X^*}$ are pseudomonotone operators. Then the operator $A_1+A_2\colon X\to 2^{X^*}$
	is pseudomonotone.
\end{proposition}
    
\indent In what follows we introduce the notion of coercivity.
\begin{definition}
	Let $X$ be a real Banach space and $A\colon X\to 2^{X^*}$ be an operator. We say that $A$ is coercive if either $D(A)$ is bounded or $D(A)$ is unbounded and
	$$
	\lim_{\|v\|_X\to\infty\,\,v\in D(A)}\frac{\inf\{\dual{v^*,v}{X}\,|\,\,v^*\in A v\}}{\|v\|_X}=+\infty,
	$$
	where, recall, $D(A)=\{x\in X|\,\,A(x)\neq\emptyset\}$ is the domain of $A$.
\end{definition}
The following is the main surjectivity result for multivalued  pseudomonotone and coercive operators.

\begin{proposition}\label{prop:Bartosz5}
	Let $X$ be a real, reflexive Banach space and $A\colon X\to 2^{X^*}$ be pseudomonotone and coercive. Then $A$ is surjective, i.e., for all $b\in X^*$ there exists $v\in X$ such that $Av\ni b$.
\end{proposition}

We also recall Lemma 3.4.12 of \cite{DMP2}, known as Ehrling's Lemma.

\begin{lemma}\label{Lemma_Ehrling}
	If $V_0,\, V,\, V_1$ are Banach spaces such that $V_0\subset V\subset V_1$, the embedding of $V_0$ into $V$ is compact and the embedding of $V$ into $V_1$ is continuous, then for every $\varepsilon>0$, there exists a constant $c(\varepsilon)>0$ such that
	\begin{eqnarray}
	\|v\|_V\leq\varepsilon\|v\|_{V_0}+c(\varepsilon)\|v\|_{V_1}\quad \text{for all}\,\,\,v\in V_0.
	\end{eqnarray}
\end{lemma} 

\indent Let $X$ be a Banach space and $T>0$.
We introduce the space $BV(0,T;X)$ of functions of
bounded total variation on $[0,T]$. Let $\pi$ denotes any finite
partition of $[0,T]$ by a family of disjoint subintervals $\{\sigma_i
= (a_i,b_i)\}$ such that $[0,T] = \cup_{i=1}^n \bar{\sigma}_i$.
Let $\mathcal{F}$ denote the family of all such partitions. Then,
for a function $x\colon [0,T]\to X$ and for $1 \leq q < \infty$,
we define a seminorm 
\begin{equation}\nonumber
\| x \|^q_{BV^q (0,T;X)} = \sup_{\pi \in \mathcal{F}} \left\{
\sum_{\sigma_i \in \pi} \|x(b_i)-x(a_i)\|_X^q \right\},
\end{equation}
and the space
$$
BV^q(0,T;X)=\{x\colon [0,T]\to X |\,\,\| x \|_{BV^q (0,T;X)}<\infty\}.
$$
For $1 \leq p \leq \infty$, $1\leq q < \infty$ and Banach spaces $X$, $Z$ such that $X
\subset Z$, we introduce a vector space
$$ M^{p,q}(0,T;X,Z) = L^p(0,T;X) \cap BV^q (0,T;Z).$$
Then $M^{p,q}(0,T; X,Z)$ is also a Banach space with the norm given by $\| \cdot \|_{L^p(0,T;X)} + \| \cdot \|_{BV^q(0,T;Z)}.$\\

The following proposition will play the crucial role for the convergence of the Rothe functions which will be constructed later. 
For its proof, we refer to \cite{Kalita2013}. 

\begin{proposition}\label{prop:Bartosz6}
	Let $1\leq p,q<\infty$. Let $X_1\subset X_2\subset X_3$ be real Banach spaces such that $X_1$ is reflexive, the embedding $X_1\subset X_2$ is compact and the embedding $X_2\subset X_3$ is continuous. Then the embedding $M^{p,q}(0,T;X_1;X_3)\subset L^p(0,T;X_2)$ is compact. 
\end{proposition}

The following version of Aubin-Celina  convergence theorem (see \cite{Aubin}) will be used in what follows.
\begin{proposition}\label{prop:Bartosz7}
	Let $X$ and $Y$ be Banach spaces, and $F\colon X\to 2^Y$ be a multifunction such that\\[-.6cm]
	\begin{itemize}
		\item[(a)] the values of $F$ are nonempty, closed and convex subsets of $Y$,\\[-.6cm]
		\item[(b)] $F$ is upper semicontinuous from $X$ into $w-Y$.
	\end{itemize}
	Let $x_n\colon(0,T)\to X$, $y_n\colon(0,T)\to Y$, $n\in \mathbb{N}$, be measurable functions such that $x_n$ converges almost everywhere on $(0,T)$ to a function $x\colon(0,T)\to X$ and $y_n$ converges weakly in $L^1(0,T;Y)$ to $y:(0,T)\to Y$. If $y_n(t)\in F(x_n(t))$ for all $n\in\mathbb{N}$ and almost all $t\in (0,T)$, then $y(t)\in F(x(t))$ for a.e. $t\in (0,T)$.
\end{proposition}

We conclude this section with a well known Young's inequality
\begin{equation}\label{Young}
ab\leq\varepsilon a^2+\frac{1}{4\varepsilon}b^2,
\end{equation}
for all $a,b\in\real$, $\varepsilon>0$.

\section{Problem formulation}\label{Sec_3}
In this section we formulate an abstract evolutionary inclusion of parabolic type involving Clarke subdifferential. We also impose assumptions on the data of the problem.\\
\indent Let $V$ be a real, reflexive, separable Banach space and $H$ be a real, separable Hilbert space equipped with  the inner product $(\cdot,\cdot)_H$ and the corresponding norm given by  $\|v\|_H=\sqrt{(v,v)_H}$ for all $v\in H$. For simplicity of notation we will write  $|v|=\|v\|_H$, $(u,v)=(u,v)_H$ for all $u,v\in H$ and $\|v\|=\|v\|_V$, $\skalar{l,v}=\dual{l,v}{V}$ for all $v\in V$, $l\in V^*$.  Identifying $H$ with its dual, we assume, that the spaces $V,\, H$ and $V^*$ form an evolution triple, i.e., $V\subset H\subset V^*$ with all embeddings being dense and continuous. Moreover, we assume, that the embedding $V\subset H$ is compact. Let $i\colon V\to H$ be an embedding operator (for $v\in V$ we still denote $iv\in H$ by $v$). For all $u\in H$ and $v\in V$, we have $\skalar{u,v}=(u,v)$. We also introduce a reflexive Banach space $U$ and the operator $\iota\in {\cal L}(V,U)$. For $T>0$, we denote by $[0,T]$ a time interval and introduce the following spaces of time dependent functions: ${\cal V}=L^2(0,T;V)$, ${\cal V^*}=L^2(0,T;V^*)$, ${\cal H}=L^2(0,T;H)$, ${\cal U}=L^2(0,T;U)$, ${\cal U^*}=L^2(0,T;U^*)$, equipped with their classical $L^2$ norms. We use notation $\dual{u,v}{\cal V}=\int_0^T\skalar{u(t), v(t)}dt$, for all $u,v\in {\cal V}$, $(u,v)_{\cal H}=\int_0^T(u(t), v(t))dt$ for all $u,v\in {\cal H}$ and $\dual{u,v}{\cal U}=\int_0^T\dual{u(t), v(t)}{U}dt$ for all $u,v\in {\cal U}$. Finally we define the space ${\cal W}=\{v\in {\cal V}\,\, |\,\,v'\in {\cal V^*}\}$. Hereafter, $v'$ denotes the time derivative of $v$ in the sense of distribution. \\
\indent We consider the operator $A\colon V\to V^*$ and the functions $f\colon [0,T]\to V^*$, $J\colon U\to \real$. Using the above notation, we formulate the following problem.

\vskip3mm \noindent {\bf Problem ${\cal P}$}. {\it Find $u\in\cV$ such that $u(0)=u_0$, and
	\begin{align}
	& u'(t)+Au(t)+\iota^*\xi(t)=f(t)\quad \text{\it in}\,\,\,V^* \ \ \text{for a.e.} \ t\in (0,T)\label{3.1}
	\end{align}
	\vspace{-1mm} with
	\begin{equation}
	\xi(t)\in \partial J(\iota u(t))\quad \text{\it for\ a.e.}\
	\,t\in (0,T).\label{3.2}\end{equation} }

\noindent Now we impose assumptions on the data of Problem  $\cP$. \\

\noindent {$H(A)$} The operator $A\colon V\to V^*$ satisfies
\begin{itemize}
	\item[(i)] $A$ is pseudomonotone,
	\item[(ii)] $\|Av\|_{V^*}\leq a+b\|v\|$ for all $v\in V$, with $a\geq 0$, $b>0$,
	\item[(iii)] $\skalar{Av,v}\geq \alpha\|v\|^2-\beta|v|^2$
	for a.e. $t\in
	(0,T)$, for all $v\in V$ with $\alpha>0$, $\beta\geq 0$.
\end{itemize}

\medskip
\noindent {$H(J)$} The functional $J\colon U\to \real$ is
such that
\begin{itemize}
\item [(i)] $J$ is locally  Lipschitz\vspace{-1mm},
\item [(ii)] $\partial
J$ satisfies the following growth condition\vspace{-1mm}
\begin{align*}\| \xi \|_{U^{\ast}} \leq d(1+\|u\|_U), 
\end{align*}
for all $u\in U,\, \xi\in \partial J(u)$ with $d>0$.
\end{itemize}

\noindent {$H(\iota)$} The operator $\iota\colon V
\rightarrow U$ is linear, continuous and compact. Moreover,
there exists a Banach space $Z$ such that $V\subset Z\subset H$, where the embedding $V\subset Z$ is compact, the embedding $Z\subset H$ is continuous, and the operator $\iota$ can be decomposed as $\iota=\iota_2\circ\iota_1 $, where $\iota_1\colon V\to Z$ denotes the (compact) identity mapping and  $\iota_2\in{\cal L}(Z,U)$.  
\\


\noindent {$H(f)$} $f \in L^2(0,T;V^*)$.\\

\noindent {$H(0)$}  $u_0 \in H$. 


\medskip
\medskip

In the rest of the paper we always assume that assumptions $H(A)$, $H(J)$, $H(\iota)$, $H(f)$ and $H(0)$ hold.

\begin{corollary}\label{Cor_3.1}
	For every $\varepsilon>0$ there exists $c_{\iota}(\varepsilon)>0$ such that
	\begin{align}
\nonumber	\|\iota u\|_U\leq\varepsilon\|u\|_V+c_{\iota}(\varepsilon)|u|\quad\text{for all}\,\,\,u\in V.
	\end{align}
\end{corollary}
\begin{proof}	
	It follows from $H(\iota)$ that $\|\iota u\|_U=\|\iota_2\circ \iota_1 u\|_U\leq \|\iota_2\|\|\iota_1 u\|_Z$. On the other hand, from Lemma \ref{Lemma_Ehrling}, we have $\|\iota_1 u\|_Z\leq \varepsilon\|u\|+c(\varepsilon)|u|$, which yields the assertion.
\end{proof}
\medskip 

The direct consequence of Corollary \ref{Cor_3.1} is following.
\begin{corollary}\label{Cor_3.2}
	For every $\varepsilon>0$ we have
	\begin{align}\label{3.4}
	\|\iota u\|^2_U\leq\varepsilon\|u\|^2+\bar{c}_\iota(\varepsilon)|u|^2,
	\end{align}
	where $\bar{c}_\iota(\varepsilon)=2c_\iota^2(\sqrt{\varepsilon/2})$.
\end{corollary}

\section{The Rothe problem}\label{Sec_4}

In this section we consider a semidiscrete approximation of Problem $\cP$ known as Rothe problem. Our goal is to study a solvability of the Rothe problem, to obtain a-priori estimates for its solution and to study the convergence of semidiscrete solution to the solution of the original problem as the discretization parameter converges to zero.\\
We start with a uniform division of the time interval. Let $N\in \nat$ be fixed and $\tau=T/N$ be a time step. In the rest of the paper we denote by $c$ a generic positive constant independent on discretization parameters, that can differ from line to line. 
Let $u_\tau^0\in V$ be given and assume that
\begin{equation}\label{0} 
u_\tau^0\to u_0 \ \mbox{strongly in} \ H
`\end{equation}
and
\begin{equation}\label{-1}
\|u_\tau^0\|\leq\frac{c}{\sqrt{\tau}}.
\end{equation}
We define the sequence $\{\fntau\}_{n=0}^N$ by the formula  
  
\begin{align*}
&f^1_\tau:=\frac{1}{\tau}\int_0^\tau f(t)dt,\\
&\fntau:= \frac{3}{2\tau}\int_{(n-1)\tau}^{n\tau} f(t)\,dt-\frac{1}{2\tau}\int_{(n-2)\tau}^{(n-1)\tau} f(t)\,dt,\quad
n=2,\ldots,N.
\end{align*}

We now formulate the following Rothe problem.

\bigskip \noindent{\bf Problem ${\cal P}_\tau$}. {\it Find a
	sequence\ $\{u^n_\tau\}_{n=1}^N\subset V$  such that
	\begin{align}
	 \frac{1}{\tau}(u_\tau^1-u_\tau^0)+&Au_\tau^1+\iota^*\xi_\tau^1=f_\tau^1\label{1},\\[2mm]
	 &\xi_\tau^1\in\partial J(\iota u_\tau^1)\label{1a}
	 	\end{align}}
and for $n=2,...,N$ 
\begin{align}
\frac{1}{\tau}\left(\frac{3}{2}\untau-2\untauminus+\frac{1}{2}\untauminuss\right)+&A\untau+\iota^*\xintau=\fntau,\label{2}\\[2mm]
&\xintau\in\partial J(\iota\untau).\label{2a}
\end{align} 	
	
Now we formulate the existence result for Problem $\cP_\tau$.

\begin{theorem}\label{Th_4.1}
	 There exists $\tau_0>0$ such that for all $0<\tau<\tau_0$, Problem $\cP_\tau$ has a solution.   
\end{theorem}
\begin{proof}
We can formulate (\ref{1}) as: find $u_\tau^1\in V$ such that 
\begin{equation}\label{eq1}
\tau f_\tau^1+u_\tau^0\in T_1(u_\tau^1),
\end{equation}
where $T_1\colon V\to 2^{V^*}$ is defined as $T_1(v)=i^*iv+\tau Av+\tau\iota^*\partial J(\iota v)$ for all $v\in V$, where, recall, $i\colon V\to H$ is the identity mapping, and $i^*$ denotes its adjoint operator. We observe, that $T_1$ is pseudomonotone. In fact, the operator $V\ni v\to i^*iv\in V^*$ is pseudomonotone, as it is linear and monotone (cf. \cite{Zeidler1990}). Moreover, the operator $\tau A$ is pseudomonotone by assumption $H(A)$(i) and the operator $\tau\iota^*\partial J(\iota v)$ is pseudomonotone by by assumption $H(J)$(ii) and by Proposition \ref{Prop_2.1}. Hence, the operator $T_1$ is pseudomonotone by Proposition \ref{Prop_sum_pseudo}.\\
In order to check coercivity, let $\eta\in T_1v$, which means, that $\eta=i^*iv+\tau Av+\tau\iota^*\xi$, whith $\xi\in\partial J(\iota v)$. Then, using $H(A)$(ii) and (\ref{3.4}), we have
\begin{align*}
&\skalar{T_1v,v}=|v|^2+\tau\skalar{Av,v}+\tau\dual{\xi,\iota v}{U}\geq\\[2mm]
 &\left(1-\tau\beta-\tau\left(\frac12+d\right)c_{\iota}(\varepsilon)\right)|v|^2+\tau\left(\alpha-\left(\frac12+d\right)\varepsilon\right)\|v\|^2-\frac12d^2. 
\end{align*}    
Taking $\varepsilon$ and $\tau$ small enough, we find out that $\skalar{T_1v,v}\geq c_1\tau\|v\|^2-c_2$ with $c_1, c_2>0$.  Then, it follows that $T_1$ is coercive. Hence, by Proposition \ref{prop:Bartosz5} operator $T_1$ is surjective for $\tau< \tau_0$, where $\tau_0>0$ and problem (\ref{eq1}) has a solution. \\
Next we proceed recursively. For $n=2,...,N$ having $\untauminus$ and $\untauminuss$ we formulate (\ref{2}) as: find $\untau\in V$, such that
\begin{align}\label{eq2}
\frac{2}{3}\tau\fntau+\frac{4}{3}\untauminus-\frac{1}{3}\untauminuss\in T_2(\untau),
\end{align} 
where $T_2\colon V\to 2^{V^*}$ is defined as $T_2(v)=i^*iv+\frac{2}{3}\tau Av+\frac{2}{3}\tau\iota^*\partial J(\iota v)$ for all $v\in V$.
Proceeding analogously as in the first part of the proof, we can easy show that $T_2$ satisfies assumptions of Proposition \ref{prop:Bartosz5} for $\tau$ small enough, which guaranties the solvability of (\ref{eq2}) and completes the proof. 
\end{proof}\\

Now we provide a result on a-priori estimates for the solution of Problem $\cP_\tau$.  We denote by $c$ a generic positive constant independent of $\tau$ that may differ from line to line. Moreover, if $c$ depends on $\varepsilon$, we write $c(\varepsilon)$. \\

\begin{lemma}\label{Lemma_apriori}
Let $\{\untau\}_{n=0}^N$ be a solution of Problem $\cP_\tau$. Then we have
\begin{align}
&\label{3}\tau \sum_{n=0}^N \|\untau\|_V^2\le c,\\[2mm]
&\label{4}\max_{n=0,...,N}|\untau|_H<c,\\[2mm]
&\label{5}\tau \sum_{n=1}^N \left\|\xintau\right\|_{U^*}^2\le c,\\[2mm]
&\label{6}\tau \left\|\frac{u_\tau^1-u_\tau^0}{\tau}\right\|_{V^*}^2 \le c,\\[2mm]
&\label{7}\tau \sum_{n=2}^N\left\|\frac{1}{\tau}(\frac{3}{2}\untau -2\untauminus +\frac{1}{2}\untauminuss)\right\|_{V^*}^2 \le c,\\[2mm]
&\label{7.5}\sum_{n=2}^N \left|\untau-2\untauminus+\untauminuss\right|^2\le c.
\end{align}
\end{lemma}
\begin{proof}
We test \eqref{1} with $u_\tau^1$ to obtain
\[
(\uone-\uzero,\uone)+ \tau \langle A\uone,\uone\rangle + \tau \langle \iota^*\xi_\tau^1,\uone\rangle = \tau\langle f_\tau^1,\uone\rangle.
\]	
Now we deal with every term of resulting equality in the following way
\begin{align}
&\label{8}(\uone-\uzero,\uone)=\frac{1}{2} |\uone|^2 -\frac{1}{2}|\uzero|^2 + \frac{1}{2}|\uone-\uzero|^2,\\[2mm]
&\label{9}\tau \langle A\uone,\uone\rangle \ge \tau \alpha \|\uone\|^2 -\tau\beta |\uone|^2,\\[2mm]
&\label{10}\tau \langle \iota^*\xi_{\tau}^1,\uone\rangle =\tau \langle \xi_{\tau}^1,\iota\uone\rangle_U \ge - \tau \|\xi_{\tau}^1\|_{U^*} \|\iota \uone\|_U \ge -\tau d(1+\|\iota \uone\|_{U})\|\iota\uone\|_U \ge \\[2mm]
&\nonumber -\tau\left(\frac{1}{2}+d\right)\varepsilon\|\uone\|^2 -\tau\left(\frac{1}{2}+d\right)c_{\iota}(\varepsilon) |\uone|^2 - \tau\frac{1}{2}d^2,\\[2mm]
&\label{11}\tau\langle f_\tau^1,\uone\rangle \le \varepsilon\tau \|\uone\|^2 + \tau c(\varepsilon) \|f_\tau^1\|_{V^*}^2. 
\end{align}
It is clear that (\ref{9}) follows from $H(A)$(iii). Moreover, (\ref{10}) is a consequence $H(J)$(ii) and Corollary \ref{Cor_3.2}.  Finally, (\ref{11}) follows from (\ref{Young}) with $c(\varepsilon)=1/4\varepsilon$.  Summing up \eqref{8}--\eqref{11} we get
\begin{align}
&\nonumber\left(\frac{1}{2} -\tau\beta -\tau\left(\frac{1}{2} +d\right)c_{\iota}(\varepsilon)\right)|\uone|^2 + \tau \left(\alpha-\varepsilon\left(\frac{3}{2}+d\right)\right)\|\uone\|^2   \\[2mm]
&\label{12}+ \frac{1}{2}|\uone-\uzero|^2\le \frac{1}{2}|\uzero|^2 + \tau\left(\frac{1}{2}d^2+c(\varepsilon)\|f_\tau^1\|_{V^*}^2\right).
\end{align}
Now we test \eqref{2} with $\untau$ and get
\[
\left(\frac{3}{2}\untau-2\untauminus +\frac{1}{2}\untauminuss,\untau\right) + \tau \langle A\untau,\untau\rangle + \langle\tau\xintau,\iota\untau\rangle = \tau \langle \fntau,\untau\rangle.
\]
We deal with each terms of the last equality as follows
\begin{align}\label{13}
\qquad\qquad\quad\,\,\nonumber\left(\frac{3}{2} \untau-2\untauminus +\frac{1}{2}\untauminuss,\untau\right) =& \frac{1}{4}\left(|\untau|^2 + |2\untau-\untauminus|^2-|\untauminus|^2\right.
\\[2mm]
&-\left.|2\untauminus-\untauminuss|^2 + |\untau-2\untauminus+\untauminuss|^2\right),
\end{align}
\begin{align}
&\label{14}\tau \langle A\untau,\untau\rangle \ge \tau \alpha \|\untau\|^2-\tau\beta|\untau|^2,\\[2mm]
& \label{15}\tau\langle \iota^*\xintau,\untau\rangle \ge -\tau\left(\frac{1}{2}+d\right)\varepsilon\|\untau\|^2 -\tau\left(\frac{1}{2} +d\right)c_{\iota}(\varepsilon)|\untau|^2-\tau\frac{1}{2}d^2,\\[2mm]
&\label{16}\tau \langle \fntau,\untau\rangle \le \varepsilon \tau \|\untau\|^2 + c(\varepsilon)\tau \|\fntau\|_{V^*}^2.
\end{align}
We sum up \eqref{13}--\eqref{16} and get
\begin{align}
&\nonumber\frac14|\untau|^2 +\frac14|2\untau-\untauminus|^2 +\frac14|\untau-2\untauminus+\untauminuss|^2 + \\[2mm]
& \nonumber\tau\left(\alpha-\left(\frac32+d\right)\right)\|\untau\|^2\le 
\frac14|\untauminus|^2 +\frac14|2\untauminus-\untauminuss|^2+ \\[2mm]
&\label{17} \tau\left(\beta+\left(\frac12 +d\right)c_{\iota}(\varepsilon)\right)|\untau|^2+\tau\frac12d^2+c(\varepsilon)\tau \|\fntau\|_{V^*}^2.
\end{align}
Taking $\tau$ and $\varepsilon$ small enough we get from \eqref{12}
\[
|\uone|^2 + \tau\|\uone\|^2 +|\uone-\uzero|^2\le c\left(|\uzero|^2+\tau\|f_\tau^1\|_{V^*}^2\right).
\]
As $\uzero \to u_0$ in $H$, $\uzero$ is bounded. Moreover, $\tau\|f_\tau^1\|^2_{V^*}$ is bounded, hence we have
\begin{equation}
\label{18}|\uone|^2 +|\uone-\uzero|^2 + \tau \|\uone\|^2\le c.
\end{equation}
We multiply \eqref{12} by $\frac12$ and, after simple reformulation, we conclude that
\begin{align}
&\nonumber\frac14|\uone|^2+\frac12 \tau\left(\alpha-\varepsilon\left(\frac32+d\right)\right)\|\uone\|^2 + \frac14|\uone-\uzero|^2\le \\[2mm]
&\label{19}\frac14|\uzero|^2 + \frac14\tau d^2 + \frac12 c(\varepsilon)\tau\|f_\tau^1\|^2_{V^*} +\frac12\tau\left(\beta+\left(\frac12+d\right)c_{\iota}(\varepsilon)\right)|\uone|^2.
\end{align}
We replace $n$ by $k$ in (\ref{17}) and sum up \eqref{19} with \eqref{17} for $k=2,\ldots,n$. Hence we obtain
\begin{align}
&\nonumber\frac14 |\untau|^2 +\frac14|2\untau-\untauminus|^2 +\frac14 |\uone-\uzero|^2+\frac14 \sum_{k=1}^k |\uktau-2\uktauminus+\uktauminuss|^2  \\[2mm]
&+\nonumber\frac12 \sum_{k=1}^n \tau\left(\alpha-\varepsilon\left(\frac32+d\right)\right)\|\untau\|^2 \le \frac14 |\uzero|^2 +\frac14|2\uone-\uzero|^2 \\[2mm]
&\label{20} + \frac12 \tau c(\varepsilon) \sum_{k=1}^n \|\fktau\|^2_{V^*} + \frac12 n\tau c^2 +\sum_{k=1}^n\tau \left(\beta+\left(\frac12+d\right)c_{\iota}(\varepsilon)\right)|\uktau|^2.
\end{align}
From \eqref{20} we conclude that get 
\begin{align*}
|\untau|^2&\le |\uzero|^2+|2\uone-\uzero|^2\\[2mm]
& +2\tau c(\varepsilon)\sum_{k=1}^n \|\fktau\|^2_{V^*} + 2Tc^2+4\sum_{k=1}^n\tau\left(\beta+\left(\frac12+d\right)c_{\iota}(\varepsilon)\right)|\uktau|^2.
\end{align*}
For $\tau$ small enough we apply Gronwall Lemma and get
\begin{equation}
\label{21}\max_{n=1,...,N} |\uktau|^2 \le c\left(|\uzero|^2 + |2\uone-\uzero|^2 + \tau\sum_{k=1}^N \|\fktau\|^2_{V^*} +T\right).
\end{equation}
By (\ref{0}) $|\uzero|$ is bounded  and by \eqref{18} we have
\begin{equation}
\label{22}|2\uone-\uzero|^2=|\uone+(\uone-\uzero)|^2\le 2(|\uone|^2+|\uone-\uzero|^2) \le c.
\end{equation}
Moreover 
\begin{equation}
\label{23}\tau\sum_{k=1}^N \|\fktau\|^2_{V^*} \le c\|f\|^2_{L^2(0,T;V^*)}.
\end{equation}
Hence, by \eqref{21} and boundedness of $|\uzero|$ we get \eqref{4}. Applying \eqref{22}, \eqref{23}, \eqref{4} and \eqref{-1} in \eqref{20} we get \eqref{3}. By $H(J)$(ii) and \eqref{3} we get \eqref{5}. 
From \eqref{1}, $H(A)$(ii), $H(J)$(ii) and \eqref{3}
\begin{align}
\nonumber\tau \left\|\frac{\uone-\uzero}{\tau}\right\|^2_{V^*} &\le c\tau \left(\|f_\tau^1\|^2_{V^*} + \|A\uone\|^2_{V^*} + \|\iota^*\xi_{\tau}^1\|^2_{V^*}\right) \\[2mm]
&\label{24}
\le c\tau\left(\|f_\tau^1\|^2_{V^*}+ a^2 +b^2\|\uone\|^2 + d\|\iota\|^2_{{\cal L}(V,U)}(1+\|\uone\|^2)\right)\le c,
\end{align}
which proves (\ref{6}). Moreover, from \eqref{2}, using again $H(A)$(ii), $H(J)$(ii), we get
\begin{align}\label{25}
&\nonumber\tau \sum_{n=2}^N \left\|\frac{1}{\tau}(\frac32 \untau-2\untauminus+\frac12\untauminuss)\right\|_{V^*}^2\\
& \le c\left(\tau \sum_{n=2}^N \|\fntau\|^2_{V^*} +\left(b^2+\|\iota\|^2_{{\cal L}(U,V)}d^2\right)\tau \sum_{n=2}^N \|\untau\|^2 + \left(a^2+\|\iota\|^2_{{\cal L}(U,V)}d^2\right)T\right).
\end{align}
Applying (\ref{3}) and (\ref{23}) in (\ref{25}) we complete the proof of \eqref{7}. Finally, \eqref{7.5} follows from \eqref{0}, \eqref{4}, \eqref{20}, \eqref{22}, \eqref{23}.
\end{proof}\\

Now we provide another lemma. 

\begin{lemma}\label{Lemma_u0u1} The following convergence holds 
	\[
	\uone-\uzero\to 0 \  \ \mbox{strongly in} \ H\, \mbox{as} \ \tau \to 0.
	\]
\end{lemma}
\begin{proof}
For the proof it is enough to show, that for any subsequence of the sequence $\{\uone-\uzero\}$ one can find a subsequence, which converges to $0$ strongly in $H$. Suppose than, that  $\{\uone-\uzero\}$ is any subsequence of the original sequence (denoted by the same symbol for simplicity). From \eqref{18} we know that $\{\uone-\uzero\}$ is bounded in $H$, hence there exists $\eta\in H$, such that for a subsequence (again denoted by the same symbol), there holds  
	\begin{equation}
	\label{26}\uone-\uzero \to \eta \ \ \mbox{weakly in} \ H.
	\end{equation}	
From \eqref{6} we have	
\begin{align*}
& \|\uone-\uzero\|_{V^*}^2\le \tau c \to 0.
\end{align*}	
Hence $\uone-\uzero\to 0$ strongly in $V^*$, hence also weakly in $V^*$. On the other hand, from (\ref{26}), $\uone-\uzero\to \eta$ weakly in $H$, hence also weakly in $V^*$. From the uniqueness of weak limit in $V^*$ we have $\eta=0$, hence 
\begin{equation}
\label{29}\uone-\uzero \to 0 \ \ \mbox{weakly in} \ H.
\end{equation}
In what follows we will use the identity 
\begin{equation}\label{29a}
|\uone-\uzero|^2 = (\uone-\uzero,\uone-\uzero)=(\uone-\uzero,\uone)-(\uone-\uzero,\uzero).
\end{equation}
After some reformulation, we have from \eqref{1} 
\begin{align}
\nonumber(\uone-\uzero,\uzero)&=(\uone-\uzero,\uone)-|\uone-\uzero|^2\le (\uone-\uzero,\uone)\\[2mm]
\label{30}&=\tau \left(\langle f_\tau^1,\uone\rangle -\langle A\uone,\uone\rangle - \langle \xi_\tau^1,\iota\uone\rangle_{U^*\times U}\right).
\end{align}
From \eqref{0} and \eqref{29} we have
\begin{equation}
\label{31}(\uone-\uzero,\uzero)\to 0.
\end{equation}
Now we estimate the right-hand side of \eqref{30} in two steps. First we get
\begin{equation}
\label{32}\tau \langle f_\tau^1,\uone\rangle \le \tau \|f_\tau^1\|_{V^*}\|\uone\|=(\tau\|f_\tau^1\|^2_{V^*})^{1/2}(\tau\|\uone\|^2)^{1/2}.
\end{equation}
By Jensen's inequality
\begin{equation}
\nonumber\tau \|f_\tau^1\|_{V^*}=\tau\left\|\frac{1}{\tau}\int_0^\tau f(t)\, dt\right\|_{V^*}^2 \le \frac{1}{\tau} \left(\int_0^\tau \|f(t)\|_{V^*}\, dt\right)^2\le \int_0^\tau \|f(t)\|_{V^*}^2\, dt \ \to 0.
\end{equation}
and by \eqref{3} $\label{34}\tau\|\uone\|^2\le c$. 
Hence, from \eqref{32} we have 
\begin{equation}
\label{35}\tau \langle f_\tau^1,\uone\rangle \to 0.
\end{equation}
as $\tau\to 0$. Then, using $H(A)$(iii) and a slight modification of (\ref{3.4}) we get
\begin{align*}
&-\tau (\langle A\uone,\uone\rangle- \langle \xi_\tau^1,\iota\uone\rangle_{U^*\times U})\\[2mm]
&\leq \tau\left( -(\alpha-\varepsilon)\|\uone\|^2+(\beta+c(\varepsilon))|\uone|^2+c \right)\leq  \tau(\beta+c(\varepsilon)|\uone|^2+c)
\end{align*}
whenever $\varepsilon<\alpha$. 
From \eqref{4} $|\uone|$ is bounded, hence 
\begin{equation}
\label{36}-(\langle A\uone,\uone\rangle- \langle \xi_\tau^1,\iota\uone\rangle_{U^*\times U}) \to 0.
\end{equation}
From \eqref{30},\eqref{31},\eqref{35},\eqref{36} and the squeeze theorem we get
\begin{equation}
\nonumber(\uone-\uzero,\uone)\to 0.
\end{equation}
Combining it with (\ref{29a}) we obtain the thesis.
\end{proof}

\medskip

\noindent Basing on the solution of the Rothe problem ${\cal P}_\tau$  we define the following functions $\bar{u}_{\tau}, u_\tau\colon [0,T]\to V$, $\bar{\xi}_\tau\colon[0,T]\to U^*$ and $\bar{f}_\tau\colon[0,T]\to V^*$.
\[
\bar{u}_\tau(0)=\uzero  \ \mbox{and} \ \bar{u}_\tau(t)=\uktau \ \ \mbox{for} \ t\in ((n-1)\tau, n\tau], \ n=1,\ldots,N,
\]
\[
u_\tau(t)= \begin{cases}
\frac32\uone-\frac12 \uzero +(\uone-\uzero)\frac{t-\tau}{\tau}, \ \mbox{for} \ t\in [0,\tau] \\[2mm]
\frac32\untau-\frac12\untauminus + (\frac32\untau -2\uktauminus +\frac12\untauminuss)\frac{t-n\tau}{\tau}, \ \ t\in [(n-1)\tau,n\tau],\,n=2,\dots,N,
\end{cases}
\]
\[
\bar{\xi}_\tau(t)= \xiktau \ \mbox{for} \ t\in((n-1)\tau, n\tau], \ n=1,\ldots,N,
\]
\[
\bar{f}_\tau(t) = \fktau \ \mbox{for} \ t\in((n-1)\tau, n\tau], \ n=1,\ldots,N.
\]
\begin{lemma}
	We have the following convergence result:
	\begin{equation}
	\label{40}\|u_\tau - \bar{u}_\tau\|_{L^2(0,T;V^*)} \to 0.
	\end{equation}
\end{lemma}
\begin{proof}
	We observe that
	\[
	u_\tau - \bar{u}_\tau = 
	\begin{cases}
	(\uone - \uzero)\frac{t-\half\tau}{\tau} \ \mbox{for} \ t\in (0,\tau] \\[3mm]
	(\frac32\untau -2\untauminus +\frac12\untauminuss)\frac{t-(n-\half)\tau}{\tau}-\frac14 (\untau-2\untauminus+\untauminuss) \ \mbox{for} \ t\in((n-1)\tau, n\tau]
	\end{cases}	
	\]	
Hence, we conclude
\begin{align}
& \nonumber\|u_\tau-\bar{u}_\tau\|_{L^2(0,T;V^*)}^2 = \int_0^\tau \|u_\tau-\bar{u}_\tau\|_{V^*}^2 \, dt + \sum_{n=2}^N \int_{(n-1)\tau}^{n\tau} \|u_\tau - \bar{u}_\tau\|_{V^*}^2\, dt\le \\[2mm]
&\nonumber\int_0^\tau \|\uone-\uzero\|_{V^*}^2 \left|\frac{t-\half \tau}{\tau}\right|^2\, dt + 2\sum_{n=2}^N \int_{(n-1)\tau}^{n\tau} \left\|\frac32 \untau-2\untauminus + \frac12\untauminuss\right\|_{V^*}^2\left|\frac{t-(n-\frac12)\tau}{\tau}\right|^2\, dt \\[2mm]
& \nonumber+2 \cdot\frac{1}{16} \sum_{n=2}^N \int_{(n-1)\tau}^{n\tau} \|\untau-2\untauminus +\untauminuss\|_{V^*}^2\,dt =\frac{1}{12} \tau  \|\uone-\uzero\|_{V^*}^2 \\[2mm]
&\nonumber + \frac{1}{6} \tau \sum_{n=2}^N  \left\|\frac32 \untau-2\untauminus + \frac12\untauminuss\right\|_{V^*}^2 + \frac18\|\iota^*\| \tau \sum_{n=2}^N \left|\untau-2\untauminus +\untauminuss\right|^2. 
\end{align}
Hence, the result follows from \eqref{6}--\eqref{7.5}.
\end{proof}

\begin{lemma}\label{Lemma_apriori_2}
	The functions $\bar{u}_\tau, u_\tau$ and $\bar{\xi}_\tau$ satisfy:
	\begin{align}
	&\label{42}\|\bar{u}_\tau\|_{L^2(0,T;V)} \le c,\\[2mm]
	&\label{43}\|u_\tau\|_{L^2(0,T;V)}\le c,\\[2mm]
	&\label{44}\|\bar{u}_\tau\|_{L^\infty(0,T;H)}\le c,\\[2mm]
	&\label{45}\|u_\tau\|_{L^\infty(0,T;H)}\le c,\\[2mm]
	&\label{46}\|\bar{\xi}_\tau\|_{L^2(0,T;U^*)}\le c,\\[2mm]
	&\label{47}\|\bar{u}_\tau\|_{M^{2,2}(0,T;V;V^*)}\le c,\\[2mm]
	&\label{47.5}\|u_\tau'\|_{L^2(0,T;V^*)}\le c.
	\end{align}	
\end{lemma}
\begin{proof}
	The proof of \eqref{42}, \eqref{44}, \eqref{45} and \eqref{46} follows directly from \eqref{3}--\eqref{5}. Now we prove \eqref{43}. Using \eqref{3} we estimate
	\begin{align*}
	&\|u_\tau\|_{L^2(0,T;V)}^2 = \int_0^T \|u_\tau(t)\|^2\,dt = \int_0^\tau \left\|\frac32 \uone-\frac12 \uzero + (\uone-\uzero)\frac{t-\tau}{\tau}\right\|^2\, dt + \\[2mm]
	&\sum_{n=2}^N \int_{(n-1)\tau}^{n\tau} \left\|\frac32\untau-\frac12 \untauminus + \left(\frac32 \untau-2\untauminus + \frac12 \untauminuss\right)\frac{t-n\tau}{\tau}\right\|^2\, dt \le\\[2mm]
	& c\left(\tau \|\uone\|^2 + \tau \|\uzero\|^2 + \tau\sum_{n=2}^N (\|\untau\|^2 + \|\untauminus\|^2 + \|\untauminuss\|)\right) \le c\tau \sum_{n=0}^N \|\untau\|^2\ \le c.	
	\end{align*}
Now we pass to the proof of \eqref{47}. Taking into account \eqref{42} it is enough to estimate $\|\bar{u}_\tau\|_{BV^2(0,T;V^*)}$. We know from \cite{Kalita2013}, that
\begin{equation}
\label{48}\|\bar{u}_\tau\|_{BV^2(0,T;V^*)}\le T\tau \sum_{i=1}^N \left\|\frac{\uitau-\uitauminus}{\tau}\right\|_{V^*}^2.
\end{equation}
Denote be $\delta_i=\uitau-\uitauminus, \ i=1,\ldots,N$, then
\[
\sum_{i=1}^N \|\delta_i\|_{V^*}^2\le \sum_{i=1}^N \left\|\frac{3}{2}\delta_i\right\|_{V^*}^2\le \left\|\frac32\delta_1\right\|_{V^*}^2 + 2\sum_{i=2}^N \left\|\frac32 \delta_i -\frac12\delta_{i-1}\right\|_{V^*}^2 + 2\sum_{i=1}^{N-1} \left\|\frac12\delta_i\right\|_{V^*}^2.
\]
Hence it follows that
\[
\frac12 \sum_{i=1}^{N-1}\|\delta_i\|_{V^*}^2 + \|\delta _N\|_{V^*}^2\le \frac94\|\delta _1\|_{V^*}^2 + 2\sum_{i=2}^N \left\|\frac32\delta_i-\frac12 \delta_{i-1}\right\|_{V^*}^2.
\]
Hence
\[
\sum_{i=1}^N \|\delta_i\|_{V^*}^2 \le c\left(\|\delta_1\|_{V^*}^2 + \sum_{i=2}^N \left\|\frac32\delta_i-\frac12\delta_{i-1}\right\|_{V^*}^2\right),
\]
which implies
\begin{align}
&\label{49}\tau\sum_{i=1}^N\left\|\frac{\uitau-\uitauminus}{\tau}\right\|_{V^*}^2 = \tau\sum_{i=1}^N\left\|\frac{\delta_i}{\tau}\right\|_{V^*}^2\\[2mm]
&\nonumber\le c\left(\tau\left\|\frac{\uone-\uzero}{\tau}\right\|^2_{V^*} + \sum_{i=2}^N\tau \left\|\frac1\tau (\frac32(\uitau-\uitauminus) -\frac12(\uitauminus  -u_{\tau}^{i-2})\right\|^2_{V^*} \right)
 \\[2mm]
&\nonumber=c\left(\tau\left\|\frac{\uone-\uzero}{\tau}\right\|^2_{V^*} +\sum_{i=2}^N\tau \left\|\frac1\tau(\frac32 \uitau - 2\uitauminus +\frac12 u_\tau^{i-2}  )\right\|_{V^*}^2\right).
\end{align}
Applying \eqref{6}, \eqref{7} and \eqref{49} in \eqref{48} completes the proof of \eqref{47}.

\noindent Finally we prove \eqref{47.5}. To this end we observe that
\begin{equation}
\label{50}u_\tau'(t) = \begin{cases}
\frac1\tau(\uone-\uzero) \ \ \mbox{for} \ t\in(0,\tau]\\[2mm]
\frac1\tau(\frac32\untau-2\untauminus + \frac12 \untauminuss) \ \ \mbox{for} \ t\in ((n-1)\tau, n\tau], \,\,n=2,\dots,N.
\end{cases}
\end{equation}
Now
\begin{equation}
\nonumber\|u_\tau'\|_{L^2(0,T;V^*)}^2 = \tau\left\|\frac{\uone-\uzero}{\tau}\right\|_{V^*}^2 + \tau \sum_{n=2}^N
\left\|\frac1\tau\left(\frac32 \untau - 2\untauminus +\frac12 u_\tau^{n-2}  \right)\right\|_{V^*}^2.
\end{equation}
From \eqref{6} and \eqref{7} we get \eqref{47.5}.
\end{proof}

\

At the end of this section, we observe, that due to \eqref{50}, relations \eqref{1}--\eqref{2a} are equivalent with 
\begin{align}
& \label{51}u_\tau'(t) + A\bar{u}_\tau(t) + \iota^* \bar{\xi}_\tau(t) = \bar{f}_\tau(t)  \ \ \mbox{for a.e.} \ t\in (0,T),\\[2mm]
& \label{52}\bar{\xi}_\tau (t)\in \partial J(\iota \bar{u}_\tau(t))  \ \ \mbox{for a.e.} \ t\in (0,T).
\end{align}

\section{Convergence of Rothe method}\label{Sec_5}
In this section we study a convergence of the functions introduced in the previous section to a solution of Problem ${\cal P}$.
For that purpose, we will provide  their convergence to some element in space $\cal V$, and then, passing to the limit in \eqref{51} and \eqref{52} we will show, that the limit element solves the original problem.  
First of all we introduce the Nemytskii operators $\mathcal{A}\colon \cV\to \cV^*$, $\bar{\iota}\colon \cV\to \cU$ corresponding to $A$ and $\iota$, respectively, given by 
\[
(\mathcal{A}v)(t) = Av(t), \ \mbox{and} \ (\bar{\iota}v)(t) = \iota v(t)\qquad \mbox{for all} \ v\in \cV.
\]
Let $\bar{\iota}^*\colon \cU^*\to\cV^*$ denote the adjoint operator to $\bar{\iota}$. Then \eqref{51} and \eqref{52} are equivalent with
\begin{align}
&\label{59}u_\tau'(t) + (\mathcal{A}\bar{u}_\tau)(t) + (\bar{\iota}^*\bar{\xi}_\tau)(t) = \bar{f}_\tau(t), \ \mbox{for a.e.} \ t\in (0,T),\\[2mm]
&\label{60}\bar{\xi}(t)\in \partial J((\bar{\iota}\bar{u}_\tau)(t))  \ \mbox{for a.e.} \ t\in (0,T).
\end{align}

\noindent In what follows we will use the following pseudomonotonicity property of operator ${\cal A}$ (see Lemma 2 in \cite{Kalita}).

\begin{lemma}\label{lemma:Bartosz2}
	Assume that $H(A)$ holds and a sequence $\{v_n\}\subset {\cal V}$ satisfies: $v_n$ is bounded in $M^{p,q}(0,T;V,V^*)$, $v_n\to v$ weakly in ${\cal V}$ and $\limsup_{n\to\infty}\dual{{\cal A}v_n,v_n-v}{\cal V}\leq 0$. Then ${\cal A}v_n\to {\cal A}v$ weakly in ${\cal V^*}$. 
\end{lemma}

\noindent Now we are in a position to provide the main convergence result. 
\begin{theorem}\label{theorem_1}
	There exists a function $u\in\mathcal{W}$ such that 
	\begin{align}
	& \label{53} u_\tau \to u \ \mbox{weakly in} \ \mathcal{W} \ \mbox{and weakly* in} \ L^\infty(0,T;H), \\[2mm]
	&\label{54}\bar{u}_\tau \to u \ \mbox{weakly in} \ \mathcal{V} \ \mbox{and weakly* in} \ L^\infty(0,T;H).
	\end{align}
Furthermore $u$ is a solution of Problem ${\cal P}$.
\end{theorem}
\begin{proof}
	From the bounds obtained in Lemma~\ref{Lemma_apriori_2} there exist $u, u_1\in \cV$, $u_2\in\cV^*$ and $\xi\in\cU^*$, such that
	\begin{align}
	&\label{55}\bar{u}_\tau \to u \ \mbox{weakly in} \ \mathcal{V} \ \mbox{and weakly* in} \ L^\infty(0,T;H), \\[2mm]
	&\label{56}u_\tau \to u_1 \ \mbox{weakly in} \ \mathcal{V} \ \mbox{and weakly* in} \ L^\infty(0,T;H), \\[2mm]
	&\label{57}u_\tau'\to u_2 \ \mbox{weakly in} \ \mathcal{V}^*,\\[2mm]
	&\label{58}\bar{\xi}_\tau \to \xi \ \mbox{weakly in} \ \mathcal{U}^*.
	\end{align}
	
A standard argument shows that $u_1'=u_2$. From \eqref{40} we have $u_\tau-\bar{u}_\tau \to 0$ strongly in $\mathcal{V}^*$,  hence 	$u_\tau-\bar{u}_\tau \to 0$ weakly in $\mathcal{V}^*$. In particular, from \eqref{55} and \eqref{56} $\bar{u}_\tau \to u$  weakly in $\mathcal{V}^*$ and $u_\tau\to u_1$ weakly in $\mathcal{V}^*$, hence $\bar{u}_\tau-u_\tau\to u-u_1$ weakly in $\mathcal{V^*}$. From uniqueness of weak limit we have $u-u_1=0$, so $u=u_1$.\\
Now we will show that $u$ is a solution of Problem ${\cal P}$. To this end we deal with initial condition first. From \eqref{56} and \eqref{57} we have $u_\tau\to u$ weakly in $\cW$. Since the embedding $\cW\subset C(0,T;H)$ is continuous, it follows that for a subsequence $u_\tau\to u$ weakly in $C(0,T;H)$, so in particular 
\begin{align}\label{58_a}
u_\tau(0)\to u(0) \ \mbox{weakly in} \ H.
\end{align}
On the other hand, by the definition of function $u_\tau$, we have
\begin{align*}
u_\tau(0)=\frac12u_\tau^1+\frac12 u_\tau^0=\frac12(u_\tau^1-u_\tau^0)+u_\tau^0.
\end{align*}
Hence, by Lemma \ref{Lemma_u0u1} and by \eqref{0}, we get $u_\tau(0)\to u_0$ strongly in $H$, so also weakly in $H$. Comparing it with (\ref{58_a}) we conclude, from uniqueness of the weak limit in $H$ that
\begin{align}\label{58_b}
u_\tau(0)\to u_0=u(0) \ \mbox{strongly in} \ H.
\end{align}

It follows from \eqref{59} that
\begin{equation}
\label{61}\langle u_\tau',v\rangle_{\cV^*\times \cV} + \langle \mathcal{A}\bar{u}_\tau, v\rangle_{\cV^*\times \cV} + \langle \bar{\iota}^*\bar{\xi}_\tau, \bar{\iota}v \rangle_{\cV^*\times \cV}= \langle \bar{f}_\tau, v \rangle_{\cV^*\times \cV} \ \mbox{for all} \ v\in \cV.
\end{equation}
We pass to the limit with \eqref{61} as $\tau\to 0$.
From \eqref{57} we have
\begin{equation}
\label{62}\langle u_\tau',v\rangle_{\cV^*\times \cV} \to \langle u',v\rangle_{\cV^*\times \cV}.
\end{equation}
From \eqref{58} we have
\begin{equation}
\label{63}\langle \bar{\xi},\bar{\iota}_\tau v\rangle_{\cU^*\times \cU} \to \langle \xi, v \rangle_{\cU^*\times \cU}.
\end{equation}
By standard arguments
\begin{equation}
\label{64}\langle \bar{f}_\tau, v \rangle_{\cV^*\times \cV} \to \langle f,v \rangle_{\cV^*\times \cV}.
\end{equation}
Then, we have 
\begin{equation}
\label{64.5}\bar{f}_\tau \to f \ \mbox{strongly in} \ \cV^*. 
\end{equation}
In the next step we will show, that
\begin{equation}
\label{65}\limsup \langle \mathcal{A}\bar{u}_\tau, \bar{u}_\tau-u\rangle_{\cV^*\times \cV} \le 0.
\end{equation}
From \eqref{61}
\begin{align*}
& \langle \A \bar{u}_\tau,\bar{u}_\tau -u\rangle_{\cV^*\times \cV} = \langle \bar{f}_\tau, \bar{u}_\tau-u\rangle_{\cV^*\times \cV} -\langle u_\tau', \bar{u}_\tau-u \rangle_{\cV^*\times \cV} - \langle \bar{\xi}_\tau,\bar{\iota} \bar{u}_\tau-\bar{\iota} u\rangle_{\cU^*\times \cU}.
\end{align*}
so 
\begin{align}
\nonumber\limsup \langle \A \bar{u}_\tau, \bar{u}_\tau-&u\rangle_{\cV^*\times \cV} \le \limsup \langle \bar{f}_\tau, \bar{u}_\tau-u\rangle_{\cV^*\times \cV} \\[2mm]
&\label{66}-\liminf \langle u_\tau',\bar{u}_\tau-u\rangle_{\cV^*\times \cV} - \liminf\langle \bar{\xi}_\tau, \bar{\iota}\bar{u}_\tau-\bar{\iota}u\rangle_{\cU^*\times \cU}. 
\end{align}
From \eqref{64.5} and \eqref{54} we have
\begin{equation}
\label{67}\lim \langle \bar{f}_\tau, \bar{u}_\tau-u\rangle_{\cV^*\times \cV} =0.
\end{equation}
From \eqref{47} and Proposition \ref{prop:Bartosz6} we have for a subsequence 
\begin{equation}
\label{68}\bar{\iota}_1\bar{u}_\tau\to \bar{\iota}_1 u \ \mbox{in} \ L^2(0,T;Z),
\end{equation}
where $\bar{\iota}\colon\cV\to L^2(0,T;Z)$ is the Nemytskii operator corresponding to $\iota_1$. We have
\begin{align}
&\nonumber\|\bar{\iota}\bar{u}_\tau-\bar{\iota}u\|_{\cU}^2 = \int_0^T\|\iota\bar{u}_\tau(t) -\iota u(t)\|_U^2\, dt = \int_0^T\|\iota_2\circ \iota_1(\bar{u}_\tau(t) -u(t))\|_U^2\, dt \\[2mm]
&\label{69} \le \|\iota_2\|^2_{{\cal L}(Z,U)} \int_0^T\|\iota_1(\bar{u}_\tau(t) -u(t))\|_Z^2 \, dt = \|\iota_2\|^2_{{\cal L}(Z,U)} \|\bar{\iota}_1 \bar{u}_\tau-\bar{\iota_1}u\|^2_{L^2(0,T;Z)}.
\end{align}
From \eqref{68} and \eqref{69} we have
\begin{equation}
\label{70}\bar{\iota}\bar{u}_\tau\to \bar{\iota}u \ \mbox{in} \ \cU.
\end{equation}
From \eqref{58} and \eqref{70} we have
\begin{equation}
\label{70.5}\lim\langle \bar{\xi}_\tau, \bar{\iota}\bar{u}_\tau-\bar{\iota}u\rangle_{\cU^*\times \cU} =0.
\end{equation}
Now we will show that
\begin{equation}
\label{71}\liminf \langle u_\tau', \bar{u}_\tau-u\rangle_{\cV^*\times \cV} \ge 0.
\end{equation}
First of all we will deal with two parts separately, namely
\begin{equation}
\label{72}\langle u_\tau', \bar{u}_\tau-u\rangle_{\cV^*\times \cV} = \langle u_\tau', u_\tau-u\rangle_{\cV^*\times \cV} + \langle u_\tau',\bar{u}_\tau -u_\tau\rangle_{\cV^*\times \cV}.
\end{equation}
To deal with the first term, we will use \eqref{56} and the convergence $|u_\tau(0)-u(0)|\to 0$ (see (\ref{58_b})). Hence
\begin{align}
&\nonumber\liminf \langle u_\tau', u_\tau-u\rangle_{\cV^*\times \cV} \ge \liminf \langle u_\tau'-u',u_\tau-u\rangle_{\cV^*\times \cV}\\[2mm]  
&\nonumber + \liminf \langle u',u_\tau-u\rangle_{\cV^*\times \cV}=\liminf\int_0^T \langle u_\tau'-u',u_\tau-u\rangle_{V^*\times V}\, dt\\[2mm]
&\nonumber = \liminf\left(\half|u_\tau(T)-u(T)|^2 -\half|u_\tau(0)-u(0)|^2\right)\ge \\[2mm]
&\nonumber \ge \liminf \half |u_\tau(T)-u(T)|^2 +\liminf\left(-\half|u_\tau(0)-u(0)|^2\right)\\[2mm]
&\label{73}=\liminf \half |u_\tau(T)-u(T)|^2 \ge 0. 
\end{align}
Now we deal with the second part
\begin{align}
&\nonumber\langle u_\tau',\bar{u}_\tau-u_\tau \rangle_{\cV^*\times \cV} = -\int_0^\tau \frac1\tau\left(  \uone-\uzero,\uone-\uzero \right)\frac{t-\half\tau}{\tau}\, dt  \\[2mm]
&\nonumber-\sum_{n=2}^N\int_{(n-1)\tau}^{n\tau}\frac1\tau \left(\frac32 \untau - 2\untauminus +\half \untauminuss,\frac32 \untau - 2\untauminus +\half \untauminuss\right)\frac{t-(n-\half)\tau}{\tau}\,dt \\[2mm]
&\nonumber-\sum_{n=2}^N\int_{(n-1)\tau}^{n\tau}\frac1\tau \left(\frac32 \untau - 2\untauminus +\half \untauminuss, \frac14(\untau-2\untauminus+\untauminuss) \right).
\end{align}
We observe that
\begin{equation}
\nonumber\int_0^\tau \frac1\tau\left(  \uone-\uzero,\uone-\uzero \right)\frac{t-\half\tau}{\tau}\, dt =0,\end{equation}
\begin{equation}
\nonumber\int_{(n-1)\tau}^{n\tau}\frac1\tau \left(\frac32 \untau - 2\untauminus +\half \untauminuss, \frac32 \untau - 2\untauminus +\half \untauminuss\right)\frac{t-(n-\half)\tau}{\tau}\, dt =0.
\end{equation}
Hence, 
\begin{equation}
\label{76}\langle u_\tau', \bar{u}_\tau -u_\tau\rangle_{\cV^*\times \cV} = \frac14\sum_{n=2}^N \left(\frac32 \untau-2\untauminus +\frac12\untauminuss, \untau-2\untauminus + \untauminuss\right).
\end{equation}
We now use inequality 
\begin{align*}
\left(\frac32a-2b+\frac12 c,a-2b+c\right)&=\frac32|a-b|^2-2(a-b,b-c)+\frac12|b-c|^2\\[2mm]
&\ge \frac12|a-b|^2-\frac12 |b-c|^2
\end{align*}
for all $a,b,c\in H$. Then
\begin{align}
&\nonumber\sum_{k=2}^N \left(\frac32 \uktau-2\uktauminus +\frac12\uktauminuss, \uktau-2\uktauminus + \uktauminuss\right) \ge\\[2mm]
&\nonumber \sum_{k=2}^N \frac12 \left|\uktau-\uktauminus\right|^2 -\frac12\left|\uktauminus-\uktauminuss\right|^2 = \\[2mm]
&\label{77}\frac12 \left|u_\tau^N-u_\tau^{N-1}\right|^2 - \frac12 \left|\uone-\uzero\right|^2 \ge -\frac12 \left|\uone-\uzero\right|^2.
\end{align}
Using \eqref{76}, \eqref{77} and Lemma~\ref{Lemma_u0u1} we conclude that $\liminf\langle u_\tau', \bar{u}_\tau -u_\tau\rangle_{\cV^*\times \cV}\ge 0$, which together with \eqref{72} and \eqref{73}  completes the proof of \eqref{71}. Now \eqref{65} follows from \eqref{66}, \eqref{67}, \eqref{70.5} and \eqref{71}.

From \eqref{47}, \eqref{55}, \eqref{65} and Lemma 1 of \cite{Kalita2013} we get
\begin{equation}
\label{78}\A\bar{u}_\tau \to \A u \ \mbox{weakly in} \ \cV^*.
\end{equation}
Using \eqref{62}--\eqref{64} and \eqref{78} we can pass to the limit in \eqref{61} and get 
\begin{equation}
\nonumber\langle u' + \A u + \bar{\iota^*} \xi -f, v\rangle_{\cV^*\times \cV}=0
 \end{equation}
 which, by standard technique shows that $u$ satisfies the first relation of Problem $\cal P$.\\
 It remains to pass to the limit with (\ref{60}).  
 First we recall that the multifunction $\partial J\colon U\to U^*$
 has nonempty, closed and convex values. Furthermore,
 by Proposition 5.6.10 of \cite{DMP2}, it is also upper semicontinuous from $U$ (equipped
 with the strong topology) into $U^*$ (equipped with the weak topology). Hence, from \eqref{58}, \eqref{70} and Proposition \ref{prop:Bartosz7} we have 
 \begin{equation}
 \nonumber\xi(t)\in \partial J(\iota u(t)) \ \mbox{a.e.} \ t\in (0,T),
 \end{equation}
 which completes the proof of the theorem. 
\end{proof}

\section{Example}\label{sec_example}
In this short section we provide an example of particular problem, for which our theoretical result can be applied.\\

\noindent Let $\Omega\subset\real^n$, $n\in\nat_+$, be an open bounded domain with a locally Lipschitz boundary $\partial \Omega$, that consists of two disjoint parts $\Gamma_N$ and $\Gamma_C$. Let $\nu$ denote the unit outward normal vector at the boundary $\partial \Omega$. Finally let $T>0$. With this notation we consider the following problem.

\vskip3mm \noindent {\bf Problem ${\cal Q}$}. {\it Find $u\colon [0,T] \times \Omega\to\real$ such that
\begin{align*}
\begin{cases}
u_t(t,x)-\Delta u(t,x)=f_0(t,x)\qquad\qquad  \mbox{for}\,\,\,\, (t,x)\in[0,T]\times\Omega,\\[1mm]
\frac{\partial u(t,x)}{\partial \nu}=f_N(t,x)\qquad\qquad\quad\quad\qquad\,\,\,\,  \text{for}\,\,\,\, (t,x)\in[0,T]\times\Gamma_N,\\[1mm]
-\frac{\partial u(t,x)}{\partial \nu}\in\partial j (u(t,x))\qquad\qquad\qquad\,\,\,\,\,  \text{for}\,\,\,\, (t,x)\in[0,T]\times\Gamma_C,\\[1mm]
u(0,x)=u_0\qquad\qquad\quad\qquad\quad\qquad\,\,\,\,\,\,  \text{for}\,\,\,\, x\in\Omega.
\end{cases}
\end{align*}	
}

We impose the following assumptions on the data of Problem ${\cal Q}$. \\

\noindent $H(f_0):$\,\, $f_0\in L^2((0,T)\times\Omega)$,\\[2mm]
\noindent $H(f_N):$\,\, $f_N\in L^2((0,T)\times\Gamma_N)$,\\ [2mm]
\noindent $H(j):$ \,\, $j\colon\real\to\real$ is a function such that
\begin{itemize}
	\item[(i)] $j$ is locally Lipschitz,
	\item[(ii)] $|\xi|\leq d_j(1+|s|)$ for all $\xi\in\partial j(s), s\in\real$, with $d_j>0$.
\end{itemize}
Now we pass to the variational formulation of {Problem ${\cal Q}$}. To this end we consider the spaces $H=L^2(\Omega)$, $V=H^1(\Omega)$ and $U=L^2(\Gamma_C)$ with the norms given by 
\begin{align*}
&|u|^2:=\|u\|^2_H=\int_{\Omega}|u(x)|^2\,dx\,\, \text{ for\, all}\,\,u\in H,\\
&\|u\|^2:=\|u\|^2_V=|u|^2+\int_{\Omega}|\nabla u(x)|^2\,dx \,\, \text{ for\, all}\,\,u\in V,\\
&\|u\|^2_U=\int_{\Gamma_C}|u(x)|^2\,dx\,\, \text{ for\, all}\,\,u\in U.
\end{align*}  
We keep notation from Section \ref{Sec_3} concerning scalar product $(u,v)$ in $H$, duality pairing $\skalar{u,v}$ in $V$ and $\dual{u,v}{U}$ in $U$. Let $\gamma\colon V\to U$ denote the trace operator, so  $\gamma u$ stands for the trace of $u\in V$ at the boundary $\partial\Omega$. Hence, the symbol $u$ in the boundary conditions of Problem ${\cal Q}$ should be understood in the sense of trace. \\
We introduce the operator $A\colon V\to V^*$ given by 
\begin{align*}
\skalar{Au,v}=\int_{\Omega}\nabla u(x)\cdot\nabla v(x)\,dx\quad\text{for\, all}\,\,u,v\in V 
\end{align*}  
and the function $f\colon(0,T)\to V^*$ given by 
\begin{align*}
\skalar{f(t),v}=\int_{\Omega}f_0(t,x)\cdot v(x)\,dx+\int_{\Gamma_N}f_N(t,x)v(x)\,d\Gamma\quad \text{for\,all}\,\, v\in V,\, \text{a.e.} \  t\in (0,T).
\end{align*}
With this  notation we formulate a variational problem corresponding to Problem ${\cal Q}$. It reads as follows.
\vskip3mm \noindent {\bf Problem ${\cal Q}_V$}. {\it Find $u\in \cW$  such that
	\begin{align*}
	\begin{cases}
	(u_t,v)+ \skalar{Au(t),v}+\int_{\Gamma_C}j^\circ(\gamma u(t,x);\gamma v(x))\,d\Gamma\geq\skalar{f(t),v}\quad \text{for\,\,all}\,\, v\in V,\, a.e. t\in (0,T),\\
	u(0)=u_0.
	\end{cases}
	\end{align*}	
}
Note that Problem ${\cal Q}_V$ has a form of parabolic hemivariational inequality. Formally it has been obtained from Problem ${\cal Q}$ by multiplying the first equation by a test function $v\in V$, integrating the resulting equation over $\Omega$, applying boundary conditions and using Definition \ref{Def_Clarke_subdifferential}. If a function $u$ is a solution of Problem${\cal Q}_V$, then it is said to be a weak solution of Problem ${\cal Q}$.  \\
Now we introduce a functional $J\colon U\to\real$ given by 
\begin{align*}
J(u)=\int_{\Gamma_C}u(x)\,d\Gamma\quad \text{for\, all}\,\,u\in U.
\end{align*} 
We are in a position to formulate more general problem. 
\vskip3mm \noindent {\bf Problem $\tilde{\cal Q}_V$}. {\it Find $u\in \cW$  such that
	\begin{align*}
	\begin{cases}
	u_t(t)+ Au(t)+\gamma^*\partial J(\gamma u(t))=f(t)\quad \text{for\,\,all}\,\, a.e.\, t\in (0,T).\\
	u(0)=u_0.
	\end{cases}
	\end{align*}	
}
We observe, that every solution of  Problem $\tilde{\cal Q}_V$ is also a solution of Problem ${\cal Q}_V$. In fact, let $u$ be a solution of $\tilde{\cal Q}_V$ and let $v\in V$. Then, there exists $\eta\in \cU^*$, such that 
\begin{align}\label{example_2}
(u_t,v)+ \skalar{Au(t),v}+\dual{\eta(t),\gamma v}{U}=\skalar{f(t),v}\quad\text {for a.e.}\  \ t\in (0,T).
\end{align}
and 
\begin{align*}
\eta(t)\in\partial J(\gamma u(t))\quad \text {for a.e.} \ \, t\in (0,T).
\end{align*}
Hence, by Theorem 3.47 of \cite{MOSBOOK}, it follows that for a.e. $t\in(0,T)$, $\eta(t)$ can be treated as a function $\eta(t)\colon\Gamma_C\to\real$, such that $\eta(t)\in L^2(\Gamma_C)$ and it satisfies
\begin{align}
&\label{example_1}\dual{\eta(t),\gamma v}{U}=\int_{\Gamma_C}\eta(t,x)\cdot v(x)\,d\Gamma,\\ 
&\label{example_3}\eta(t,x)\in\partial j (u(t,x)) \quad a.e.\, x\in \Gamma_C. 
\end{align}
Applying Definition, \ref{Def_Clarke_subdifferential} we get from (\ref{example_1}) and (\ref{example_3})
\begin{align*}
\dual{\eta(t),\gamma v}{U}\leq\int_{\Gamma_C}j^\circ(\gamma u(t,x);\gamma v(x))\,d\Gamma.
\end{align*}
Combining it with (\ref{example_2}), we get 
\begin{align*}
(u_t,v)+ \skalar{Au(t),v}+\int_{\Gamma_C}j^\circ(\gamma u(t,x);\gamma v(x))\,d\Gamma\geq\skalar{f(t),v}.
\end{align*}
Keeping in mind the initial condition, we claim that $u$ solves Problem  ${\cal Q}_V$. Hence, in order to get a weak solution of Problem ${\cal Q}$, it is enough to solve Problem $\tilde{\cal Q}_V$. That's why we will concentrate on solvability of the last problem. \\
First of all we notice that it corresponds to Problem $\cP$ of Section \ref{Sec_3}, with $\iota=\gamma$.  
Moreover, we will show, that Theorem \ref{theorem_1} can be applied in our case. To this end we will examine all assumptions stated in Section \ref{Sec_3}.\\
First we notice, that operator $A$ is pseudomonotone, as it is linear and monotone. Hence $H(A)$(i) holds. Moreover, an elementary calculation shows that for all $v\in V$ we have
\begin{align*}
\|Av\|_{V^*}\leq\left(\int_{\Omega}|\nabla v(x)|^2\,dx\right)^{\frac{1}{2}}\leq \|v\|.
\end{align*}     
Hence, the assumption $H(A)$(ii) holds with the constants $a=0$, $b=1$. It is also clear that for all $v\in V$ one has  $\skalar{Av,v}=\|v\|^2-|v|^2$, so $H(A)$(iii) holds with $\alpha=1$ and $\beta=1$.\\
Now we study the properties of functional $J$. By Theorem 3.47 (iii) of \cite{MOSBOOK}, it is locally Lipschitz, hence $H(J)$(i) holds. Moreover, an elementary calculation shows that $J$ satisfies $H(J)$(ii) with the constant $d=\sqrt{2}d_j\max\{1,\sqrt{|\Gamma_C|}\}$. \\
Now let us show that the trace operator $\gamma$ satisfies assumption corresponding to $H(\iota)$. Taking $Z=H^\delta(\Omega)$ with $\delta\in[\frac{1}{2},1]$, we notice that the embedding $V\subset Z$ is compact. Hence, letting $\iota_1\colon V\to Z$ be the identity mapping, we claim that it is compact. Moreover, we define an auxiliary trace mapping $\iota_2\colon Z\to U$. Then, clearly $\gamma=\iota_2\circ\iota_1$ and the assumption $H(\iota)$ holds fo $\iota=\gamma$.\\
It is also clear that assumption $H(f)$ is a consequence of $H(f_0)$ and $H(f_N)$. Now it is enough to impose $u_0\in H$ to fulfil the last assumption $H(0)$.\\ 

In this way we have shown that Theorem \ref{theorem_1} is applicable in our case, namely the weak solution of Problem ${\cal Q}$ can be approximated by means of a double step Rothe scheme described in Section \ref{Sec_4}. \\ 
        
At the end of this section we provide a simple example of a function $j\colon\real\to\real$ which satisfies assumptions $H(j)$. Let it be defined  by the formula        
\begin{align*}
j(s)=\begin{cases}
0\qquad\hspace{2.6cm} \text{for}\,\,\,s<0\\
-de^{-s}+\frac{1}{2}ds^2+d\qquad\text{for}\,\,\,s\geq 0.
\end{cases}
\end{align*}
Then the Clarke subdifferential of $j$ is given by 
\begin{align*}
\partial j(s)=\begin{cases}
0\qquad\hspace{1.65cm} \text{for}\,\,\,s<0\\
[0,d]\qquad\hspace{1.1cm} \text{for}\,\,\,s=0\\
de^{-s}+s\qquad\hspace{.5cm}\text{for}\,\,\,s\geq 0.
\end{cases}
\end{align*}
It is easy to check that the function $j$ satisfies assumptions $H(j)$(i)-(ii). In particular, the constant $d$ from the assumption $H(j)$(ii) and the constant $d$ used in the formula of $j$ coincide.

\end{document}